\documentclass{amsart}
\usepackage[english]{babel}
\usepackage[cp1250]{inputenc}
\usepackage{amssymb,amsthm,amsfonts}
\usepackage{verbatim}
\usepackage{url}

\newcommand{\RR}{\mathbb R}
\newcommand{\NN}{\mathbb N}

\newcommand{\CC}{\mathbb C}

\newcommand{\QQ}{\mathbb Q}

\newcommand{\Pos}{\operatorname{Pos}}

\newcommand{\Id}{\operatorname{Id}}
\newcommand{\I}{\operatorname{I}}
\newcommand{\M}{\mathcal{M}}
\newcommand{\T}{\mathcal{T}}

\newcommand{\h}{\mathcal{H}}
\newcommand{\G}{\mathcal{G}}
\newcommand{\kk}{\mathcal{K}}

\newcommand{\cX}{\mathcal{X}}
\newcommand{\cY}{\mathcal{Y}}
\newcommand{\cI}{\mathcal{I}}
\newcommand{\x}{\underline{x}}

\newcommand{\diag}{\operatorname{diag}}
\newcommand{\beqn}{\begin{eqnarray*}}
\newcommand{\eeqn}{\end{eqnarray*}}
\newcommand{\LL}{\mathcal{L}}
\newcommand{\B}{\operatorname{Bor}}
\newcommand{\PP}{\mathcal{P}}

\providecommand{\eps}{\epsilon}

\newtheorem{corollary}{Corollary}
\newtheorem{theorem}{Theorem}
\newtheorem{proposition}{Proposition}
\newtheorem{lemma}{Lemma}
\theoremstyle{definition}

\newtheorem{remark}{Remark}

\begin{document}
\title{Moment problems for operator polynomials}

\author{Jaka Cimpri\v c and Alja\v z Zalar}

\address{Jaka Cimpri\v c, University of Ljubljana, Faculty of Math.~and Phys.,
Dept.~of Math., Jadranska 19, SI-1000 Ljubljana, Slovenija. www page: \url{http://www.fmf.uni-lj.si/~cimpric}.} 
\email{cimpric@fmf.uni-lj.si}
\address{Alja\v z Zalar, University of Ljubljana, Faculty of Math.~and Phys., Dept.~of Math.}
\email{aljaz.zalar@student.fmf.uni-lj.si}

\date{\today}

\begin{abstract} 
Haviland's theorem states, that given a closed subset $K$ in $\RR^n$ each functional 
$L  \colon \RR[\x] \to \RR$ positive on $\Pos(K):=\left\{p\in\RR[\x] \mid p|_K\geq 0\right\}$ 
admits an integral representation by a positive Borel measure. Schmüdgen proved, that in the 
case of compact semialgebraic set $K$ it suffices to check positivity of $L$ on a preordering $T$, 
having $K$ as the non-negativity set. Further he showed, that the compactness of $K$ is equivalent 
to the archimedianity of $T$. The aim of this paper is to extend these results from 
functionals on the usual real polynomials to operators mapping from the real matrix or operator polynomials 
into $\RR, M_n(\RR)$ or $B(\kk)$. 
\end{abstract}


\subjclass[2012]{}

\maketitle

\section{Introduction}

Let $K$ be a closed subset of $\RR^d$, $d \ge 1$. The $K$-\textit{moment problem} asks 
for which multisequences $c \colon \NN^d \to \RR$ there exists a positive Borel measure $\mu$ on $K$
such that $c_\alpha = \int_K x^\alpha \, d\mu := \int_K x_1^{\alpha_1}\cdots x_d^{\alpha_d} \, d\mu$
for every $\alpha=(\alpha_1,\ldots,\alpha_d) \in \NN^d$.
A solution to this problem is given by the following result, see \cite[Theorem 3.1.2]{Mar}:

\begin{theorem}[Haviland, 1935]\label{hav}
For a linear functional $L:\RR[\underline{x}]\rightarrow \RR$ and a closed set $K$ in $\RR^d$ the following statements are equivalent:
\begin{enumerate}
\item There exists a positive Borel measure $\mu$ on $K$ such that $L(p)=\int_K p \, d \mu$ for every $p \in \RR[\x]$.
\item $L(p)\geq 0$ holds for all $p\in \RR[\underline{x}]$ satisfying $p\geq 0$ on $K$.
\end{enumerate}
\end{theorem}

\begin{remark}\label{havkomp}
If $K$ is compact, then the measure $\mu$ is unique, see \cite[Corollary 3.3.1]{Mar}.
For noncompact $K$, the question of uniqueness is highly nontrivial
and will not be discussed here, see \cite{ps1} and \cite{ps2}.
\end{remark}

Theorem \ref{hav} is not considered entirely satisfactory, because the set
$$\Pos(K):=\{p \in \RR[\x] \mid p\geq 0 \text{ on } K\}$$ is very big.
If the set $K$ is defined by finitely many polynomial inequalities, then
the condition $L(\Pos(K))\ge 0$ is equivalent to $L(T) \ge 0$ for some set
$T$ which is much smaller that $\Pos(K)$. This is the contents of Theorem \ref{sch}.

For a finite set $S=\left\{g_1,\ldots,g_k\right\}$ in $\RR[\underline{x}]$ write
$$K_S := \left\{\x\in \RR^d\mid g_1(\x)\geq 0, g_2(\x)\geq 0, \ldots, g_k(\x)\geq 0\right\}$$
and
$$M_S := \{\sigma_0+\sigma_1 g_1+\ldots+\sigma_k g_k \mid \sigma_0,\sigma_1,\ldots,\sigma_k \in \sum \RR[\x]^2\}.$$

\begin{theorem}\label{sch}
Let $S$ be a finite subset of $\RR[\underline{x}]$ such that $K_S$ be compact. 
Then there exists a finite subset $S_1$ of $\RR[\underline{x}]$ containing $S$ such that $K_{S_1}=K_S$ and
\begin{enumerate}
\item every $p \in \RR[\x]$ such that $p\vert_{K_S} >0$ belongs to $M_{S_1}$,
\item for every linear functionals $L$ on $\RR[\x]$ such that $L(M_{S_1}) \ge 0$
there exists a positive Borel measure $\mu$ on $K_S$ such that $L(p)=\int_{K_S} p \, d \mu$ for all $p \in \RR[\x]$.
\end{enumerate}
More precisely, we can take $S_1$ to be either the set $\prod S$ of all square-free products of elements from $S$ (Schm\" udgen 1991, 
see \cite{sch-psatz}, a nice refinement is \cite{jp}) or the set $S \cup \{l^2-\sum_{i=1}^d x_i^2\}$ for some $l \in \NN$ 
(Putinar 1993, see \cite{put}). 
\end{theorem}

Note that claim (2) of Theorem \ref{sch} is a consequence of claim (1) and Theorem \ref{hav}. 

The aim of this paper is to extend Theorems \ref{hav} and \ref{sch} to matrix polynomials. 
We also have some partial results (both positive and negative) for operator polynomials.

In most of the current literature, the term \textit{operator moment problem} refers to
the question of existence of integral representations for linear mappings $L \colon \RR[\x] \to B(\kk)_h$ where
$B(\kk)_h$ is the real vector space of all bounded self-adjoint operators on a Hilbert space $\kk$.
The univariate case is well-understood, see e.g.~\cite{krein} and \cite{kovalishina}.
In the multivariate case, see \cite[Theorem I.4.3]{Vas2} for a result related to our Theorem \ref{sch}.
A different kind of a moment problem is considered in \cite{Amb-Vas} where the authors study 
the question of existence of integral representations for linear functionals 
$L \colon \RR[\x] \otimes B(\h)_h \to \RR$. Here, $\RR[\x] \otimes B(\h)_h=B(\h)_h[\x]$
is the real vector space of all polynomials with coefficients from $B(\h)_h$. 
For the unit cube in $\RR^d$, their Theorem 3 extends our Theorem \ref{sch}.

In this paper, we unify both approaches by studying integral representations of linear mappings
$L \colon \RR[\x] \otimes B(\h)_h \to B(\kk)_h$. The relevant measure and integration theory
was developed in \cite{Dor1}. It is recalled and slightly modified in Section \ref{secdob}. 
In Section \ref{sechav} we prove a generalization of Theorem \ref{hav} to arbitrary $\h$ and $\kk$, 
see Theorem \ref{genhavi} and its special case Theorem \ref{haviland} for $\kk=\RR$.
In Section \ref{schsec}, we prove a generalization of Putinar's part of Theorem \ref{sch} 
to arbitrary $\h$ and $\kk$ and a generalization of Schm\" udgen's part of Theorem \ref{sch} 
to finite-dimensional $\h$ and arbitrary $\kk$, see Theorems \ref{ver1} and \ref{cim}.  Finally, in Section \ref{schex}, we show that the main step in the proof of Theorem \ref{cim}
fails for infinite dimensional $\h$ even if $\kk=\RR$.

\section{Operator-valued measures}
\label{secdob}

Let $\PP$ be a ring of sets and let $\h$ and $\kk$ be real Hilbert spaces. 
We denote by $\LL(B(\h)_h,B(\kk)_h)$ the Banach space of all bounded linear operators from $B(\h)_h$ to $B(\kk)_h$, where 
$B(\h)_h$ and $B(\kk)_h$ are the Banach spaces of all bounded self-adjoint linear operators on $\h$ and $\kk$, respectively. 
A set function 
$$m \colon \PP \to \LL(B(\h)_h,B(\kk)_h)$$ 
is a \textit{non-negative operator-valued measure} if for every $A\in B(\h)_+$ the set function 
$$m_A \colon \PP \to B(\kk)_h, \quad m_A(\Delta)=m(\Delta)(A),$$
is a positive operator-valued measure. 

\begin{remark}\label{almost-positive}
Recall from \cite[Definition 1]{Ber} that a set function 
$$E \colon \PP\to B(\kk)_h$$ is a \textit{positive operator-valued measure}, 
if it satisfies the following conditions:
\begin{enumerate}
\item[(a)] $E(\Delta)\succeq 0$ for all $\Delta\in \PP$.
\item[(b)] $E(\Delta_1 \cup \Delta_2)=E(\Delta_1)+E(\Delta_2)$ if $\Delta_1$ and $\Delta_2$ are disjoint subsets in $\PP$.
\item[(c)] If $\Delta_i$ is an increasing sequence in $\PP$ and $\Delta=\bigcup_i \Delta_i$ belongs to $\PP$ then $E(\Delta)
=\sup_i E(\Delta_i)$.
\end{enumerate}
When $\h=\RR$, we can identify $\LL(B(\h)_h,B(\kk)_h)$ with $B(\kk)_h$. In this identification
the non-negative operator-valued measure $m$ corresponds to the positive operator-valued measure $m_1$.
Therefore, positive operator-valued measures are special cases of non-negative operator-valued measures.
\end{remark}

\begin{remark}
Our definition of a non-negative operator-valued measure is similar to the following definition from \cite[p.~511]{Dor1}:
A set function $m \colon \PP\to \LL(\cX,\cY)$, where $\PP$ is a $\delta$-ring of sets and $\cX$, $\cY$ are Banach spaces, 
is called an \textsl{operator-valued measure countably additive in the strong operator topology} if for every $x\in X$ 
the set function $m_x \colon \PP \to \cY, \Delta \mapsto m(\Delta)x$, is a countably additive vector measure. 

These definitions coincide if
$\cX=B(\h)_h$ for some  Hilbert space $\h$, 
$\cY=B(\kk)_h$ for some finite-dimensional Hilbert space $\kk$, and
$m_x(\Delta) \in B(\kk)_+$ for every $x \in B(\h)_+$ and every $\Delta \in \PP$. 
The problem with infinite-dimensional $\kk$ is that the definitions of convergence of 
$m_x(\Delta_i)$ to $m_x(\bigcup_i \Delta_i)$ do not coincide. 
\end{remark}

Let $X$ be a set, $\PP$ a $\sigma$-algebra of subsets of $X$ and $m \colon \PP \to \LL(B(\h)_h,B(\kk)_h)$
a non-negative operator-valued measure. Let $\cI$ denote the set of all $\PP$-measurable real-valued functions 
on $X$ which are  $m_A$-integrable for every $A \in B(\h)_+$.
It is a real vector space and it consists at least of all bounded measurable functions. 
In particular, if $\PP=\B(X)$ (the Borel $\sigma$-algebra of $X$) then $C_c(X,\RR)\subset\cI$.

\begin{remark}
\label{eintdef}
Let $E \colon \PP \to B(\kk)_h$ be a positive operator-valued measure. 
For every $x \in \kk$ we define a positive measure 
$E_x \colon \PP \to \RR^{\geq 0}$ by $E_x(\Delta)=\langle E(\Delta)x,x\rangle$.
We say that a $\PP$-measurable function $f \colon X \to \RR$ is $E$-\textit{integrable} if
there exists a constant $K_f \in \RR$ such that 
$\int \vert f \vert \; dE_x \le K_f \Vert x \Vert^2$ for every $x \in \kk$.
(If $\Vert f \Vert_\infty < \infty$ then $K_f=\Vert E(X) \Vert \; \Vert f \Vert_\infty$ works.)
The mapping $(x,y) \mapsto \frac14  (\int f \; dE_{x+y}-\int f \; dE_{x-y} )$
is then a bounded bilinear form; see \cite[Section 5]{Ber}. 
Therefore, there exists a bounded operator $\int f\; dE \in B(\kk)_h$ such that 
$\int f\; dE_x=\langle (\int f\; dE)x,x \rangle$ for every $x \in \kk$.
\end{remark}

For every $f \in \cI$ and every  operator $A\in B(\h)_h$, we define
$$\int {f\; dm_A}:=\int {f\; dm_{A_+}}-\int {f\; dm_{A_-}}$$ 
where $A_+,A_- \in B(\h)_+$ are the positive and the negative part of $A$. Namely, $A=A_+-A_-$, $A_+A_-=A_-A_+=0$ and hence
$\left\|A_{\pm}\right\|\leq \left\|A\right\|$ (see \cite[Proposition 5.2.2(4)]{Li}).

Let $\cI \otimes B(\h)_h$ be the algebraic tensor product of $\cI$ and $B(\h)_h$ over $\RR$. 
By the universal property of tensor products, the bilinear form 
$$\cI \times B(\h)_h \to B(\kk)_h, \quad (f,A) \mapsto \int {f\; dm_A}$$
extends to a linear map 
$$\cI \otimes B(\h)_h \to B(\kk)_h, \quad
F=\sum_{i=1}^{n}f_i\otimes A_i \mapsto \int{F\;dm}:= \sum_{i=1}^{n}\int {f_i\; dm_{A_i}}.$$

We first recall the following operator-valued version of the F. Riesz representation theorem for positive functionals,
see \cite[Theorem 19]{Ber}. A positive operator-valued measure with $\PP=\B(X)$
will be called a \textit{Borel positive operator-valued measure}.

\begin{proposition}
\label{riesz}
Let $X$ be a locally compact and $\sigma$-compact metrizable space, $\kk$ a Hilbert space 
and $T \colon C_c(X,\RR) \to B(\kk)_h$ a positive bounded linear map.
Then there exists one and only one Borel positive operator-valued measure $E$ on $X$ such that $T(f)=\int f \, dE$ for every
$f \in C_c(X,\RR)$. 
\end{proposition}

Proposition \ref{dobrakov} extends Proposition \ref{riesz} from $C_c(X,\RR)$ to $C_c(X,\RR) \otimes B(\h)_h$. 
It is similar to \cite[Theorem 2]{Dor}.
The vector space $C_c(X,\RR) \otimes B(\h)_h$ can be identified with a subspace of $C_c(X,B(\h)_h)$ 
from where it inherits the supremum norm and the positive cone $C_c(X,B(\h)_+)$. 
Unlike \cite{Dor1} we will never integrate functions from $C_c(X,B(\h)_h)$ that do not belong
to $C_c(X,\RR) \otimes B(\h)_h$.

\begin{proposition}
\label{dobrakov}
Let $X$ be a locally compact and $\sigma$-compact metrizable space, $\h$ and $\kk$ Hilbert spaces and 
$L: C_c(X,\RR) \otimes B(\h)_h \rightarrow B(\kk)_h$ a positive bounded linear map.
Then there exists a unique non-negative operator-valued measure
$$m: \B(X)\rightarrow \LL(B(\h)_h, B(\kk)_h )$$ such that 
$$L(F)=\int {F\; dm}$$ 
holds for all $F\in C_c(X,\RR) \otimes B(\h)_h.$
\end{proposition}

\begin{proof}  
For every $A\in B(\h)_+$ we define an operator $L_A \colon C_c(X,\RR)\to B(\kk)_h$ by $L_A(f)=L(f\otimes A)$. 
Since $L$ is positive, it follows that $L_A(C_c(X,\RR)_+)\succeq 0$. By Proposition \ref{riesz} there exists 
a unique Borel positive operator-valued measure $E_A$ such that $L_A(f)=\int{f\;dE_A}$ for all $f\in C_c(X,\RR)$.  
Let us define a map $$m \colon \B(X)\to \LL(B(\h)_h, B(\kk)_h ), \quad m(\Delta)(A)=E_{A_+}(\Delta)-E_{A_-}(\Delta).$$
For every $f\in C_c(X,\RR)$ and $A,B\in B(\h)_+$ we have
\beqn
\int{f\; dE_{A+B}}&=& L_{A+B}(f)= L(f\otimes (A+ B)) =\\ 
&=& L(f \otimes A)+ L(f \otimes B) = L_A(f)+ L_B(f)= \\ 
&=& \int{f\; dE_{A}}+\int{f\; dE_{B}} = \int{f\; d(E_{A}+ E_B)}.
\eeqn
It follows that $E_{A+B}=E_{A}+E_B$ by the uniqueness part of Proposition \ref{riesz}.
For general $A,B \in B(\h)_h$ we deduce that $m(\Delta)(A+B)-m(\Delta)(A)-m(\Delta)(B)
= (E_{(A+B)_+}(\Delta) - E_{(A+B)_-}(\Delta)) - (E_{A_+}(\Delta) - E_{A_-}(\Delta)) - (E_{B_+}(\Delta) - E_{B_-}(\Delta))
= E_{(A+B)_+ +A_-+B_-}(\Delta)-E_{(A+B)_- +A_++B_+}(\Delta)=0$.
Therefore, $m(\Delta)$ is additive for every $\Delta\in \B(X)$.
Similarly we show that it is also homogeneous.

We claim that $m(\Delta)$ is bounded for every $\Delta\in \B(X)$. 
Pick an increasing sequence of compact $\Delta_i \in \B(X)$ such that $X=\bigcup_i \Delta_i$. 
By Urysohn's Lemma there exist functions $u_i \in C_c(X,[0,1])$ such that $u_i|_{\Delta_i}\equiv 1$.
For every $A \in B(\h)_+$ we have that $E_A(\Delta_i)=\int \chi_{\Delta_i} \; dE_A\le \int u_i \; dE_A =L_A(u_i)=L(u_i \otimes A)$ 
which implies that $\Vert E_A(\Delta_i) \Vert \le 
\Vert L \Vert \; \Vert u_i \otimes A \Vert= \Vert L \Vert \; \Vert A \Vert$.
Furthermore, $(E_A)_x(\Delta) \le (E_A)_x(X)=\sup_i (E_A)_x(\Delta_i)$ for every $x \in \kk$, which implies that
$$\Vert E_A(\Delta) \Vert\le \sup_i \Vert E_A(\Delta_i) \Vert \le \Vert L \Vert \; \Vert A \Vert.$$
For non-positive $A \in B(\h)$ we need an additional factor 2 because
$$\Vert m(\Delta)(A)\Vert \le \Vert E_{A_+}(\Delta) \Vert+\Vert E_{A_-}(\Delta)\Vert \leq 2\Vert L \Vert \Vert A \Vert.$$
Therefore, the set function $m$ is a non-negative operator-valued measure. 

To prove that $m$ is a representing measure for $L$, 
it suffices by linearity to prove that $L(f\otimes A)=\int{(f\otimes A)\, dm}$ 
for all $f\otimes A\in C_c(X,\RR) \otimes B(\h)_+$. This follows from
$$L(f\otimes A)=L_A(f)=\int{f\; dE_{A}}= \int{f\; dm_A}=\int{(f\otimes A)\, dm}.$$

The uniqueness of $m$ follows from the uniqueness of the measures $E_A$ for every $A\in B(\h)_+$.
\end{proof}

\section{Haviland's Theorem}
\label{sechav}

Theorem \ref{haviland} extends Theorem \ref{hav} to operator polynomials. Here we will restrict ourselves to $\kk=\RR$.

\begin{theorem}
\label{haviland}
For a linear map $L \colon \RR[\x] \otimes B(\h)_h \to \RR$ and a closed set $X$ in $\RR^d$,
the following are equivalent:
\begin{enumerate}
\item There exists a non-negative Borel measure $m \colon \B(X) \to \LL(B(\h)_h, \RR)$ such that
$L(F)=\int F \, dm$ for every $F \in \RR[\x] \otimes B(\h)_h$.
\item $L(F) \ge 0$ for every $F \in \RR[\x] \otimes B(\h)_h$ such that $F \succeq 0$ on $X$.
\end{enumerate}
\end{theorem}

For $\h=\RR$, this is \cite[Theorems 3.1.2 and 3.2.2]{Mar}. 

\begin{proof}
The nontrivial direction is that (2) implies (1).
Let $A_0$ be the range of the natural mapping $\;\hat{}\; \colon \RR[\x] \to C(X,\RR)$.
By (2), $\bar{L}(\hat{p} \otimes B) := L(p \otimes B)$
is a well-defined positive linear functional on $A_0 \otimes B(\h)_h$. The set
$$C'(X,\RR):=\left\{f\in C(X,\RR)\mid \exists p\in \RR[\x]: \left|f\right|\leq \left|\hat p\right|\;\text{on}\; X\right\}$$
is clearly a vector space which contains $C_c(X,\RR)$. Since $A_0$ is cofinal in $C'(X,\RR)$, also $A_0 \otimes B(\h)_h$ 
is cofinal in $C'(X,\RR)\otimes B(\h)_h$. By the M.~Riesz extension theorem, $\bar{L}$ extends (non-uniquely) 
to a positive linear functional on $C'(X,\RR)\otimes B(\h)_h$ which will also be denoted by $\bar{L}$. Note, that 
$\bar{L}|_{C_c(X,\RR) \otimes B(\h)_h}$ is bounded, since for every $F\in C_c(X,\RR) \otimes B(\h)_h$ we have
$F\preceq \left\|F\right\|_{\infty}\otimes \mathrm{Id}$ and hence 
$\bar{L}(F)\leq \bar{L}(\left\|F\right\|_{\infty}\otimes \mathrm{Id})=\bar L(1\otimes \mathrm{Id})\left\|F\right\|_{\infty}.$
By Proposition \ref{dobrakov}, there exists a non-negative operator-valued Borel measure $m \colon \B(X) \to \LL(B(\h)_h, \RR)$
such that
\begin{equation}
\tag{*}
\bar{L}(F)=\int {F\; dm}
\end{equation}
for all $F \in C_c(X,\RR) \otimes B(\h)_h$. We have to show that (*) holds for all $F\in C'(X,\RR)\otimes B(\h)_h$ 
(and hence for all $F \in A_0 \otimes B(\h)_h$). Clearly, it suffices to show that (*) holds for every
$F=f \otimes B$ where $f \in C'(X,\RR)_+$ and $B \in B(\h)_+$. 

Write $p=x_1^2+\ldots+x_n^2$. By the proof of Claim 3 of \cite[Theorem 3.2.2]{Mar} there exists 
an increasing sequence $f_i \in C_c(X,\RR)_+$ such that $0 \le f-f_i \le \frac{1}{i} (f+\hat{p})^2$ for every $i$.
Thus, $$\bar{L}(f \otimes B)=\bar{L}_B(f)=\lim_{i\to\infty} \bar{L}_B(f_i)=\lim_{i\to\infty} \int f_i \, dE_B
\overbrace{=}^{(\ast)} \int f \, dE_B = \int f \otimes B\, dm.$$
Note that in this case $E_B$ are the usual positive Borel measures. Therefore, the existence of $\int f \, dE_B$ and $(\ast)$ 
follow from the monotone convergence theorem and the fact that the sequence $\int f_i \;dE_B$ is bounded above by $\bar{L}(f\otimes B)$.
\end{proof}

\begin{remark}
\label{fdhavi}
If the Hilbert space $\h$ in Theorem \ref{haviland} is finite-dimensional, then we can identify 
$\LL(B(\h)_h, \RR)$ with $B(\h)_h$ via the trace map $\operatorname{tr}$. 
The representation $L(F)=\int F \, dm$ then reads as  $L(F)=\int \operatorname{tr} (F \, dE)$
where $E \colon \B(X) \to B(\h)_h$ is the positive operator-valued measure that corresponds to $m$
in the above identification.
\end{remark}

To obtain versions of Hamburger, Stieltjes and Hausdorff moment problems for operator polynomials, we combine 
Theorem \ref{haviland} with the following:

\begin{proposition}
\label{fejer}
For every operator polynomial $F \in \RR[x] \otimes B(\h)_h$ we have the following equivalences:
\begin{enumerate}
\item $F(a) \succeq 0$ for every $a \in \RR$ iff $F$ is a sum of hermitian squares of polynomials from
$\RR[x] \otimes B(\h)$.
\item $F(a) \succeq 0$ for every $a \in [0,\infty)$ iff $F=\sigma_0+ x \sigma_1$ where 
$\sigma_0,\sigma_1$ are sums of hermitian squares of polynomials from
$\RR[x] \otimes B(\h)$.
\item $F(a) \succeq 0$ for every $a \in [0,1]$ iff $F=\sigma_0+ x \sigma_1+(1-x) \sigma_2+x(1-x) \sigma_3$ where 
$\sigma_i$ are sums of hermitian squares of polynomials from
$\RR[x] \otimes B(\h)$.
\end{enumerate}
\end{proposition}

In the proof we use the operator version of the Fejér-Riesz theorem, see \cite{rosenblatt} in the matrix case, \cite{rosenblum} in the
operator case and \cite[Theorem 2.1]{dr} for a survey. Since $\h$ is a real Hilbert space, while the Fejér-Riesz theorem
works only for complex Hilbert spaces, we have to complexify our $\h$ to $\h_\CC$.
From the proof it will also follow, that $F$ in (1) and  $\sigma_0, \sigma_1$ in (2) can be chosen 
as a sum of at most two hermitian squares.

\begin{proof}
\begin{enumerate}
\item By assumption, $\deg F=2n$ for some $n$.
Replacing $x=\tan t$, we get
\begin{equation*}
F(x)=(\cos t)^{-2n} \tilde{F}(\cos t,\sin t)
\end{equation*}
where $\tilde{F}(u,v):=F\left(\frac{v}{u}\right)u^{2n}$ is homogeneous and  $\tilde{F}\succeq 0$ on $\RR^2$. Clearly,
\begin{equation*}
\tilde{F}(\cos t,\sin t)=u(e^{2it})
\end{equation*}
for some operator Laurent polynomial $u$, i.e., $u(z)=\sum_{k=-n}^{n}A_k z^k$ and $A_k\in B(\h_\CC)=B(\h)_\CC$.
Since $u(e^{it})\succeq 0$ for $t\in\RR$, it follows by the Fejér-Riesz theorem that $u(e^{it})=P(e^{it})P^\ast(e^{-it})$, 
where $P$ is a usual operator polynomial, i.e., $P(z)=\sum_{k=0}^n B_k z^k$ and $B_k\in B(\h)_\CC$.
Hence 
\begin{equation*}
\tilde{F}(\cos t,\sin t) =  G(\cos t,\sin t)G^{\ast}(\cos t,\sin t),
\end{equation*}
where
\beqn
G(\cos t,\sin t)&=&P(e^{2it})e^{-itn}= \sum_{k=0}^n B_k e^{2itk-itn} =\sum_{k=0}^n B_k (e^{it})^k (e^{-it})^{n-k} =\\
&=& \sum_{k=0}^n (B_k'+iB_k'') (\cos t + i\sin t)^k (\cos t - i \sin t)^{n-k} =\\
&=& H(\cos t,\sin t)+i K(\cos t,\sin t),
\eeqn
with $B_k', B_k''\in B(\h)$ and $H,K\in \RR[u,v] \otimes B(\h)$ are homogeneous polynomials of degree $n$. 
It follows that
\begin{equation*}
\tilde{F}(\cos t,\sin t) =  H(\cos t,\sin t)H^{\ast}(\cos t,\sin t)+K(\cos t,\sin t)K^{\ast}(\cos t,\sin t).
\end{equation*}
Note that $i(-H(\cos t,\sin t)K^\ast(\cos t,\sin t)+K(\cos t,\sin t)H^\ast(\cos t,\sin t))=0$ 
since the coefficients of $\tilde{F}$ are ``real'', i.e., they belong to $B(\h)$.
Therefore,
\begin{equation*}
F(x)=H(1,x)H^{\ast}(1,x)+K(1,x)K^{\ast}(1,x).
\end{equation*}
\item From $F|_{\RR_+}\succeq 0$ it follows $G(a):=F(a^2)\succeq 0$ on $\RR$. By (1) 
$$G(a)=\sum_i P_i(a)P^{\ast}_i(a)=\sum_i(R_i(a^2)+aQ_i(a^2))(R^\ast_i(a^2)+aQ^{\ast}_i(a^2))=$$ 
$$\sum_{i}R_i(a^2)R_i^{\ast}(a^2) + a\sum_i (Q_i(a^2)R^{\ast}_i(a^2)+R_i(a^2)Q^{\ast}_i(a^2))+a^2\sum_i Q_i(a^2)Q_i^{\ast}(a^2) $$
Since $G(a)=G(-a)$ we get $$F(a^2)=G(a)=\frac{1}{2}\left(\sum_{i}R_i(a^2)R_i^{\ast}(a^2)+a^2\sum_i Q_i(a^2)Q_i^{\ast}(a^2)\right)$$ and with substitution $t=a^2$ 
the result follows. 
\item The proof is the same as in the matrix case, see \cite[Theorem 2.5]{ds} or \cite[Section 7]{sav-sch} .
\end{enumerate}
\end{proof}

Now we can explicitly formulate Hamburger's, Stieltjes' and Hausdorff's theorems for matrix polynomials.

\begin{corollary}
\label{matrixham}
Let $L$ be a linear functional on $\RR[x] \otimes S_n(\RR)$. For each $p\in \NN_0$  write 
$S_{p}:=[L(x^{p}E_{k,l})]_{k,l=1,\ldots,n}$ where $E_{k,l}$ are coordinate matrices. Then 
\begin{enumerate}
\item $L$ has an integral representation (in the sense of Remark \ref{fdhavi}) with a positive operator-valued measure $E$
whose support is contained  in $\RR$ iff $\left[S_{i+j}\right]_{i,j=0,\ldots,m}$ is positive semidefinite for every 
$m\in \NN_0$,
\item $L$ has an integral representation  with a positive operator-valued measure $E$ whose support is contained in  $[0,\infty)$
iff $\left[S_{i+j}\right]_{i,j=0,\ldots,m}$ and $\left[S_{i+j+1}\right]_{i,j=0,\ldots,m}$ are positive semidefinite for every 
$m\in \NN_0$, 
\item $L$ has an integral representation  with a positive operator-valued measure $E$ whose support is contained in  $[0,1]$
iff $\left[S_{i+j}\right]_{i,j=0,\ldots,m}$, $\left[S_{i+j+1}\right]_{i,j=0,\ldots,m}$, $\left[S_{i+j}-S_{i+j+1}\right]_{i,j=0,\ldots,m}$ and 
$\left[S_{i+j+1}-S_{i+j+2}\right]_{i,j=0,\ldots,m}$ are positive semidefinite for every $m\in \NN_0$.
\end{enumerate}
\end{corollary}

The operator version of Corollary \ref{matrixham} is less straightforward. 
For the Hamburger's theorem one has to require that for
every $m\in\NN_0$ and every tuple of operators $(A_0,\ldots,A_m) \in B(\h)$, the matrix 
$$\left[L\left(x^{i+j}A_i^\ast A_j\right) \right]_{i,j=0,\ldots,m}$$ is positive semidefinite.
For the Stieltjes' theorem we require that for
every $m\in\NN_0$ and every tuple of operators $(A_0,\ldots,A_m) \in B(\h)$, the matrices
$$\left[L\left(x^{i+j}A_i^\ast A_j\right) \right]_{i,j=0,\ldots,m}\;\mathrm{and}\;\left[L\left(x^{i+j+1}A_i^\ast A_j\right) \right]_{i,j=0,\ldots,m}$$
are positive semidefinite, while for Hausdorff's theorems we additionaly require that
$$\left[L\left(\left(x^{i+j}-x^{i+j+1}\right) A_i^\ast A_j\right)\right]_{i,j=0,\ldots,m}\; \mathrm{and}\;\left[L\left(\left(x^{i+j+1}-x^{i+j+2}\right) A_i^\ast A_j\right) \right]_{i,j=0,\ldots,m}$$
are positive semidefinite.

\medskip

The problem with the extension of Theorem \ref{haviland} to $\kk \ne \RR$ is that 
M. Riesz extension theorem is known to fail in general. However, if the mapping $L$ is completely positive
then we can use the following version of Arveson's extension theorem.

\begin{proposition}
\label{arveson}
Suppose $(E,K_1(E),K_2(E),\ldots)$ is a real matrix ordered vector space.
Let $E_0$ be a cofinal subspace of $E$. Let $\kk$ be a real Hilbert space and
$L \colon E_0 \rightarrow B(\kk)_h$ a completely positive map from the matrix ordered space $E_0$ to $B(\kk)_h$. 
Then there exists a completely positive map $L' \colon E\rightarrow B(\kk)_h$ such that $L'|_{E_0}=L$.
\end{proposition}

Proposition \ref{arveson} is very similar to \cite[Theorem 3.7.]{Pow}.
The differences are that our $E$ and $E_0$ are real vector spaces with trivial involution 
instead of complex vector spaces with general involution and that 
the codomain of our $L$ is bounded operators instead of (not necessarily bounded) sesquilinear forms.
We advice the reader to consult \cite[Section 11.1]{schbook} before continuing. 

\begin{proof} 
If $L=0$, put $L'=0$. Assume that $L\neq 0$. By Zorn's Lemma we may assume that $E=\RR x_0 \oplus E_0$ for some $x_0 \in E\setminus E_0$.
We consider the real $\ast$-vector space $\kk\otimes \kk$ with involution $(k_1\otimes k_2)^\ast=k_2\otimes k_1$.
Let  $G$ be the real vector space $(\kk\otimes \kk)_h\oplus \RR$ and let $C$ be the convex hull of  elements 
$$\left(\sum_{j,l=1}^n \alpha_{jl}k_l\otimes k_j, \sum_{j,l=1}^n \left\langle L(x_{jl}) k_l, k_j\right\rangle\right)
\in (\kk\otimes \kk)\oplus \RR$$
where $\alpha_{jl} \in \RR$, $x_{jl} \in E_0$ and $k_j \in \kk$ are such that 
$[\alpha_{jl} x_0+ x_{jl}]_{jl}\in K_n(E).$ 
It follows that $\alpha_{lj}=\alpha_{jl}$ and $x_{lj}=x_{jl}$ for every $j,l=1,\ldots,n$, hence $C \subseteq G$.

Next, we show that $(0,1)$ is an algebraic interior point of $C$ - i.e., for every
$(y,\lambda) \in G$ we will find $\delta>0$ such that $\gamma(y,\lambda)+(0,1) \in C$ for every $\gamma \in (0,\delta)$.
Since $L\neq 0$ and  $E_0$ is cofinal in $E$, there exist $x\in K_1(E_0)$, $k\in \kk$, such that $\left\langle L(x)k,k\right\rangle>0$.
Hence $(0,\left\langle L(x)k,k\right\rangle)\in C$ and with scaling we conclude $(0,\alpha)\in C$ for every $\alpha>0$.
Suppose that $y=\sum_{j,l=1}^n \alpha_{jl}k_l\otimes k_j$ where $[\alpha_{jl}]_{jl}\in M_n(\RR)_h$  and
$k_1,k_2,\ldots,k_n\in\kk$. Since $E_0$ is cofinal in $E$, there exist $z_{jl} \in K_1(E_0)$, $j,l=1,\ldots,n$,
such that $z_{jl} \pm \alpha_{jl} x_0 \in K_1(E)$. Set
$[x_{jl}]_{jl}:= \sum_j E_{jj}^T z_{jj} E_{jj}+\sum_{j<l} (E_{jj}+E_{jl})^T z_{jl}(E_{jj}+E_{jl})
+\sum_{j<l} (E_{jj}+E_{ll})^T z_{jl}(E_{jj}+E_{ll}) \in K_n(E_0)$
where $E_{jl}$ are coordinate matrices.
Clearly, $[\alpha_{jl} x_0+ x_{jl}]_{jl}=
[\alpha_{jl}]_{jl} x_0+ [x_{jl}]_{jl}= \sum_j E_{jj}^T (z_{jj}+\alpha_{jj}x_0) E_{jj}+
\sum_{j<l} (E_{jj}+E_{jl})^T (z_{jl}+\alpha_{jl}x_0)(E_{jj}+E_{jl})
+\sum_{j<l} (E_{jj}+E_{ll})^T (z_{jl}-\alpha_{jl} x_0)(E_{jj}+E_{ll}) \in K_n(E)$.
Write $\lambda_1:= \sum_{j,l=1}^n \left\langle L(x_{jl})k_l,k_j\right\rangle\geq 0$
and note that $(y,\lambda_1) \in C$. For every $0<\gamma<\min\left\{\frac{1}{\vert \lambda-\lambda_1\vert}, 1\right\}=:\delta$ we have
$\gamma(y,\lambda)+(0,1)=\gamma(y,\lambda_1)+(1-\gamma)\left(0,\frac{\gamma(\lambda-\lambda_1)+1}{1-\gamma} \right)\in C.$

On the other hand, $(0,0)$ is not an algebraic interior point in $C$. The proof is the same as in the
complex case, see \cite[Theorem 11.1.5]{schbook}. (Namely, if $(0,-\eps) \in C$ for some $\eps>0$
then we get a contradiction after a short computation.)

Now the separation theorem for convex sets, see e.g. \cite[Ch.~IV, Theorem 3.3]{Con},
gives us a linear functional $f \colon G \to \RR$ such that $f(C)\ge 0$. Since $(0,1)$ is 
in the interior of $C$, we have that $f((0,1))>0$, so we may assume that $f((0,1))=1$.
We claim that the bilinear form $M(k_1,k_2):=\frac12 f((k_1 \otimes k_2+k_2 \otimes k_1,0))$
is bounded. Namely, since $E_0$ is cofinal in $E$ we can pick $z \in K_1(E_0)$ such that
$z \pm x_0 \in K_1(E)$. By the definition of $C$, it follows that 
$(\pm k \otimes k,\langle L(z)k,k\rangle)\in C$ for every $k \in \kk$,
which implies that $\pm M(k,k)+\langle L(z)k,k\rangle=
\pm f((k \otimes k,0))+\langle L(z)k,k\rangle f((0,1))\ge 0$ for every $k \in \kk$.
Since $L(z)$ is bounded, the polarization identity implies that $M$ is also bounded.
By \cite[Ch.~II, Theorem 2.2]{Con}, there exists $L_0(x_0) \in B(\kk)_h$ such that
$\langle L_0(x_0)k_1,k_2)=M(k_1,k_2)$ for every $k_1,k_2 \in \kk$.

The mapping $L' \colon \RR x_0+E_0 \to B(\kk)_h$, $L'(\alpha x_0+z):=\alpha L_0(x_0)+L(z)$ clearly extends $L$.
To show that $L'$ is completely positive, pick any $n \in \NN$, $X \in K_n(E)$ and $k_1,\ldots,k_n \in \kk$.
Clearly, $X=[\alpha_{jl} x_0+x_{jl}]_{jl}$ for some $[\alpha_{jl}]_{jl} \in M_n(\RR)_h$ and $[x_{jl}]_{jl} \in M_n(E)_h$.
If $y=\sum_{j,l=1}^n \alpha_{jl} k_l \otimes k_j$ and 
$\lambda =\sum_{j,l=1}^n \langle L(x_{jl})k_l,k_j \rangle$ then
$$\sum_{j,l=1}^n \langle (L' \otimes \Id_{M_n(\RR)})(X)k_l,k_j \rangle=
\sum_{j,l=1}^n \langle L'(\alpha_{jl} x_0+x_{jl})k_l,k_j \rangle=$$
$$=\sum_{j,l=1}^n \alpha_{jl} \langle L(x_0)k_l,k_j \rangle+\sum_{j,l=1}^n \langle L(x_{jl})k_l,k_j \rangle=
f((y,0))+\lambda=f((y,\lambda)).$$
Since $(y,\lambda) \in C$, we have that $f((y,\lambda)) \ge 0$ which implies the claim.
\end{proof}

Theorem \ref{genhavi} is a generalization of Theorem \ref{haviland}. 
It is also a generalization of \cite[Proposition 2.1]{Sch}, where the author studies the case $\h=\CC$.

\begin{theorem}
\label{genhavi}
If $\h,\kk$ are Hilbert spaces, $X$ is a closed set in $\RR^d$  
and $$L \colon \RR[\x] \otimes B(\h)_h \to B(\kk)_h$$ is a linear map such that 
$$L \otimes \Id_{M_n(\RR)}(G) \succeq 0$$ for every integer $n \in \NN$ 
and every symmetric polynomial $G \in \RR[\x] \otimes B(\h)_h \otimes M_n(\RR)$ 
such that $G(a) \succeq 0$ for every $a \in X$, then there exists 
a non-negative Borel measure $$m \colon \B(X) \to \LL(B(\h)_h, B(\kk)_h)$$
such that  for every $F \in \RR[\x] \otimes B(\h)$
$$L(F)=\int F \, dm.$$
\end{theorem}

\begin{proof}
With the notation from the proof of Theorem \ref{haviland}, we have that $E_0=A_0 \otimes B(\h)_h$
is cofinal in $E=C'(X,\RR)\otimes B(\h)_h$ where $K_n(E)$ consists of all elements of $M_n(E)$
which are positive semidefinite in every point of $X$. Furthermore, the mapping 
$\bar{L} \colon E_0 \to B(\kk)_h$ defined by $\bar{L}(\hat{p} \otimes B) := L(p \otimes B)$ 
is completely positive by assumption. By Proposition \ref{arveson}, there exists a
completely positive extension of $\bar{L}$ to $E$. As in the proof of Theorem \ref{haviland}, the restriction
of $\bar{L}$ from $E$ to $C_c(X,\RR) \otimes B(\h)_h$ is bounded. 
By Proposition \ref{dobrakov}, it has the desired integral representation.

It remains to show that this integral representation also works on $E$. 
By linearity, it suffices to take $F=f \otimes B$ where $f\in C'(X,\RR)_+$ and $B \in B(\h)_+$ are arbitrary.
Let $p$ and $f_i$ be as in the proof of Theorem \ref{haviland} and let $x \in \kk$ be arbitrary. Then
$$\langle \bar{L}(F)x,x\rangle =\langle \bar{L}_B(f)x,x\rangle= \lim_{i\to\infty} \langle\bar{L}_B(f_i)x,x\rangle.$$
Since $\bar{L}_B(f_i)=\int f_i \, dE_B$, it follows by the monotone convergence theorem that 
$$\lim_{i\to\infty} \langle\bar{L}_B(f_i)x,x\rangle= \lim_{i\to\infty} \int f_i \, d(E_B)_x= \int f \, d(E_B)_x.$$
It follows that $f$ is $E_B$-integrable (with $K_f=\Vert \bar{L}(F) \Vert$; see Remark \ref{eintdef}).
Therefore,
$$\int f \, d(E_B)_x=\langle (\int f \, dE_B)x,x \rangle = \langle (\int F\, dm)x,x \rangle.$$
Since $x$ was arbitrary, we have that $\bar{L}(F)=\int F\, dm$ as claimed.
\end{proof}

\begin{remark}
If $X$ is compact, we can replace the complete positivity assumption in Theorem \ref{genhavi} with the weaker 
positivity assumption, see Theorem \ref{ver1} below. This can also be done if $\h=\RR$ and $\dim \kk<\infty$
and $X$ is either $\RR$ or $[0,\infty)$, see \cite{Zag,Zag1}.
\end{remark}

\section{Schm\" udgen's theorem}
\label{schsec}

Let $\h$ be a Hilbert space. A subset $\M \subseteq \RR[\x] \otimes B(\h)_h$ is a \textit{quadratic module} 
if $\Id_{\h}\in \M$, $\M+\M\subseteq \M$ and $A^\ast \M A\subseteq \M$ for every $A\in \RR[\x] \otimes B(\h)$.
The smallest quadratic module which contains a given subset $\G$ of $\RR[\x] \otimes B(\h)_h$
will be denoted by $\M_{\G}$. 
For $\h=\RR$ we get the  definition of a quadratic module in $\RR[\x]$.

A quadratic module $\M$ in $\RR[\x] \otimes B(\h)_h$ is \textit{archimedean} 
if for every operator polynomial $F \in \RR[\x] \otimes B(\h)_h$ there exists a number $n \in \NN$ such that 
$n \cdot \Id_{\h}\pm F \in \M$. If $M$ is an archimedean quadratic module in $\RR[\x]$
then the set $M'$ which consists of all finite sums
of elements of the form $mA^TA$ where $m \in M$ and $A\in \RR[\x] \otimes B(\h)$ is clearly
an archimedean quadratic module in $\RR[\x] \otimes B(\h)_h$. 

Theorem \ref{ver1} is an operator version of the Putinar's part of Theorem \ref{sch}.

\begin{theorem} \label{ver1}
Let $L: \RR[\x] \otimes B(\h)_h \rightarrow B(\kk)_h$ be a linear operator, $M \subseteq \RR[\x]$ an archimedean quadratic module and
$K_M:=\{\x\in \RR^d \mid p(\x) \succeq 0 \;\text{for all} \;p \in M\}$.
Then the following statements are equivalent:
\begin{enumerate}
\item There exists a unique non-negative operator-valued measure
$$m: \B(K_M)\rightarrow \LL(B(\h)_h, B(\kk)_h ),$$ such that 
$$L(F)=\int_{K_M} {F\, dm}$$ 
holds for all $F\in \RR[\x] \otimes B(\h)_h.$
\item $L(mA^TA)\succeq 0$ for every $m\in M$ and $A\in \RR[\x] \otimes B(\h)$ (i.e., $L(M') \succeq 0$).
\end{enumerate}
\end{theorem}

For an archimedean quadratic module $M$ in $\RR[\x]$ we define a set
$\overline{M'}=\{F \in R[\underline{x}] \otimes B(\h)_h \mid \eps + F \in M'\; \text{for all} \;\eps>0\}$.
In the sequel, we will need the following version of the Scherer-Hol theorem, 
which is a special case of \cite[Theorem 12]{Cim}. 

\begin{proposition}
\label{sh}
Let $M$ be an archimedean quadratic module in $\RR[\x]$ and $\h$ a Hilbert space.
For every element $F \in \RR[\x] \otimes B(\h)_h$, the following are equivalent:
\begin{enumerate}
\item $F \in \eps+M'$ for some real $\eps>0$.
\item For every $a \in K_M$ we have that $F(a) \succ 0$.
\end{enumerate}
\end{proposition}

\begin{proof}[Proof of Theorem \ref{ver1}] Clearly, (1) implies (2).
Suppose now that (2) is true. Our plan is to extend $L$ to a positive bounded linear map 
from $C(K_M,\RR) \otimes B(\h)_h$ to $B(\kk)_h$ and then apply Proposition \ref{dobrakov}.
This will prove that (1) is true. Recall that the norm and the positive cone of 
$C(K_M,\RR) \otimes B(\h)_h$ are inherited from $C(K_M,B(\h)_h)$, i.e., $\Vert F \Vert
=\sup_{a \in K_M} \Vert F(a) \Vert$ and $F \ge 0$ iff $F(a) \succeq 0$ for every $a \in K_M$.

Let $A_0$ be the range of the natural mapping $\hat{}\; \colon \RR[\x] \to C(K_M,\RR)$.
For every $F= \sum_i p_i \otimes A_i \in \RR[\x] \otimes B(\h)_h$ 
we will write $\hat{F} :=  \sum_i \hat{p_i}\otimes A_i \in C(K_M,\RR) \otimes B(\h)_h$.
We define a linear map $\bar{L} \colon A_0 \otimes B(\h)_h$ by $\bar{L}(\hat{F}):=L(F).$ 
To see that $\bar{L}$ is well-defined and positive, note that if $\hat{F} \succeq 0$ on $K_M$, 
then $F \in \overline{M'}$ by Proposition \ref{sh}. Now, (2) implies that $L(F) \succeq 0$.

Next, we show that $\bar{L}$ is bounded.
For every $v\in \kk$, where $\left\|v\right\|=1$, we define a functional
$\bar{L}_v:A_0 \to \RR$ by $\bar{L}_v(\hat{F})=\left\langle L(F)v,v \right\rangle$. 
Since $\bar{L}_v(M^\prime)\geq 0$, it follows 
that
$$\vert \bar L_v(\hat F) \vert \leq n_{M'}\left(F \right)\, \bar L_v(1),$$
where
$$n_{M'}(F)=\inf\left\{q\in \QQ^+\mid q \cdot \Id \pm F \in M^{\prime}\right\}.$$ 
It follows that 
$$
\Vert \bar{L}(\hat F) \Vert = 
\max_{\Vert v \Vert=1}\vert \bar L_v(\hat F)\vert 
\leq n_{M'}\left(F \right) \max_{\left\|v\right\|=1} \bar L_v(\hat \Id) =  
n_{M'}\left(F \right) \Vert \bar L(\hat \Id) \Vert.
$$
By Proposition \ref{sh}, $n_{M'}(F)=\Vert \hat{F} \Vert$. Hence 
$\Vert \bar{L} (\hat{F}) \Vert \le \Vert \hat{F} \Vert\Vert \bar L(\hat \Id) \Vert $ for every $\hat F$.
Therefore $\bar{L}$ is bounded.

By the Stone-Weierstrass theorem, $A_0$ is dense in $C(K_M,\RR)$. It follows that 
$A_0 \otimes B(\h)_h$ is dense in $C(K_M,\RR) \otimes B(\h)_h$. Therefore, $\bar{L}$ 
has a unique extension to a positive bounded map from $C(K_M,\RR) \otimes B(\h)_h$ to $B(\kk)_h$ by continuity.
\end{proof}

Let us recall from \cite{Cim3} that a quadratic module $\M$ in $S_n(\RR[\x])$ is a \textit{preordering} if 
the set $E_{11} \M E_{11}$ (or equivalently the set $\M \cap \RR[\x] \cdot \I_n$) is closed under multiplication.
The smallest preordering which contains a given set $\G \subseteq S_n(\RR[\x])$ will be denoted by $\T_{\G}$.
We will prove the following matrix version of the Schm\" udgen's part of Theorem \ref{sch}.

\begin{theorem}\label{cim}
Suppose that $\G = \left\{G_1, G_2, \ldots ,G_k \right\}\subseteq S_n(\RR[\x])$ are
such that the set $K_{\G} := \left\{ \x \in \RR^d \mid G_1(\x) \succeq 0, G_2(\x)\succeq 0,\ldots,G_k(\x)\succeq 0\right\}$ is compact. 
Then:
\begin{enumerate}
\item The preordering $\T_{\G}$ is an archimedean quadratic module.
\item Every $F\in S_n(\RR[\x])$ which satisfies $F(\x)\succ 0$ on $K_{\G}$ belongs to $\T_{\G}$.
\item For every Hilbert space $\kk$ and every linear map $L \colon S_n(\RR[\x]) \to B(\kk)_h$ such that
$L(\T_{\G}) \succeq 0$ there exists a unique non-negative measure
$m: \B(K_\G)\rightarrow \LL(S_n(\RR), B(\kk)_h )$ such that 
$L(F)=\int_{K_\G} {F\, dm}$ for every $F\in S_n(\RR[\x])$.
\end{enumerate}
\end{theorem}

The following special case of \cite[Proposition 5]{Cim3} will be used in the proof:

\begin{proposition}\label{central}
For every subset $\G\subseteq S_n(\RR[\x])$ there exists a subset $\tilde{\G}\subseteq \M_{\G}\cap \RR[\x] \cdot \I_n$ 
such that $K_{\G}=K_{\tilde{\G}}$. If $\G$ is finite, then $\tilde{\G}$ can also be chosen finite.
\end{proposition}

\begin{proof}[Proof of Theorem \ref{cim}]
By Proposition \ref{central}, there exist $g_1,g_2,\ldots, g_k \in \RR[\x]$ such that 
$K_{\G}=K_{\{g_1\cdot\I_n, g_2\cdot\I_n, \ldots ,g_k\cdot\I_n\}}=K_{\left\{g_1,g_2,\ldots, g_k\right\}}$
and $g_1\cdot\I_n, g_2\cdot\I_n, \ldots ,g_k\cdot\I_n \in \M_{\G}$.
Since $K_\G$ is compact, it follows by Theorem \ref{sch} that $T_{\left\{g_1,g_2,\ldots, g_k\right\}}$ 
is an archimedean preordering in $\RR[\x]$. Now $\T_{\G}$ is an archimedean because it contains 
the archimedean quadratic module $\left(T_{\left\{g_1,g_2,\ldots, g_k\right\}}\right)'$.
This proves claim (1). Claim (2) follows from claim (1) and Proposition \ref{sh}.
Claim (3) follows from claim (1) and Theorem \ref{ver1}.
\end{proof}

\section{An example}
\label{schex}

Let $\h$ be a Hilbert space. A quadratic module $\T \subseteq \RR[\x] \otimes B(\h)_h$ is a \textit{preordering} 
if for some (and hence every) rank one projector $P \in B(\h)_h$ the set $P \T P$ is closed under multiplication. 
Recall that $P$ is the form $P_u : x \rightarrow \left\langle x,u\right\rangle u$ for some $u \in \h$ of norm $1$.
Moreover, $P_{Su} =SP_u S^{\ast}$ and $P_u S P_u = \left\langle Su,u\right\rangle\,P_u$ for all $S \in \RR[\x] \otimes B(\h)$.
For a subset $\G$ of $\RR[\x] \otimes B(\h)_h$ write $\T_\G$ for the smallest preordering containing $\G$.

\begin{lemma}
Let $\G$ be a subset of $\RR[\x] \otimes B(\h)_h$ and $u$ an element of $\h$ of norm $1$. Write $\G_u$ for the set of
all finite products of elements of the form 
$$P_u S^\ast G S P_u=\langle GSu,Su\rangle P_u$$
where $G \in \G \cup \{\Id\}$ and $S \in\RR[\x] \otimes B(\h)$.
Then $$\T_\G=\M_{\G \cup \G_u}.$$
\end{lemma}

\begin{proof}
The inclusion $\M_{\G \cup \G_u}\subseteq \T_\G$ is clear. To prove the opposite inclusion, it suffices to show that
the quadratic module $\M_{\G \cup \G_u}$ is a preordering. Every element $F \in \M_{\G \cup \G_u}$ is of the form
$F= \sum_i R_i^\ast G_i R_i +\sum_j S_j^\ast H_i S_j$ where $G_i \in \G\cup \Id$, $H_j \in \G_u$, $R_i,S_j \in \RR[\x] \otimes B(\h)$
and both sums are finite. It follows that $P_u F P_u=\sum_i P_u R_i^\ast G_i R_i P_u +\sum_j P_u S_j^\ast P_u H_i P_u S_j P_u
=\sum_i P_u R_i^\ast G_i R_i P_u +\sum_j  H_i P_u S_j^\ast P_u^2 S_j P_u$ is a finite sum of elements from $\G_u$. 
Therefore, the set $P_u \M_{\G_u} P_u =\sum_{\mathrm{finite}} \G_u$ is closed under multiplication.
\end{proof}

Note that for every $f \in \RR[\x]\otimes \h$ and every $u \in \h$ of norm $1$ there exists an element 
$F \in \RR[\x]\otimes B(\h)$ such that $f=Fu$. It follows that the set $\G_u$ consists of all finite products
of elements of the form $\langle Gf,f \rangle P_u$ where $G \in \G \cup \{\Id\}$ and $f \in \RR[\x]\otimes \h$.

\subsection{Construction of a compact non-archimedean preordering}

We define polynomials 
$p_i(x)=\frac{x^3}{i}-x^2$, $i \in \NN$.
We have $K_{\left\{p_i\right\}}=\left\{0\right\}\cup [i,\infty)$. 
Let us define operator polynomial $G(x)\in \RR[x]\otimes B(\ell^2)$ as 
$$G(x)=\diag(p_1(x), p_2(x),\ldots),$$
which is equivalent to
\beqn
G&=&
x^3
\left(
\begin{array}{cccc}
1&0&0&\ldots\\
0&\frac{1}{2}&0&\ldots\\
0&0&\frac{1}{3}&\ldots\\
\vdots&\vdots&\vdots&\ddots\\
\end{array}
\right)
-
x^2
\left(
\begin{array}{cccc}
1&0&0&\ldots\\
0&1&0&\ldots\\
0&0&1&\ldots\\
\vdots&\vdots&\vdots&\ddots\\
\end{array}
\right)
\eeqn
We have $K_{\left\{G \right\}}=\left\{0\right\}$. 
Let $u=(1,0,0,\ldots)$. Clearly, the leading coefficient of $G$ as well as the leading coefficients of all elements from
$\{G\}_u$ are positive semidefinite operators. It follows that the leading coefficient of every element from 
$\T_{\left\{G \right\}}=\M_{\left\{G \right\} \cup \left\{G \right\}_u}$ is a positive semidefinite operator.
Therefore, $T_{\left\{G\right\}}$ does not contain $(K^2-x^2)\Id$ for any real $K$. 
It follows that the preordering $T_{\left\{G\right\}}$ is not archimedean. Moreover, the operator polynomial 
$(1-x^2)\Id$ is positive definite on $K_{\left\{G \right\}}=\left\{0\right\}$ but it does not belong to $T_{\left\{G\right\}}$.

This proves that assertions (1) and (2) of Theorem \ref{cim} do not extend from matrix polynomials to operator polynomials.
It is still an open question whether assertion (3) of Theorem \ref{cim} extends from matrix polynomials to operator polynomials.

We claim that in our example, every functional $L$ on $\RR[x] \otimes B(\ell^2)$ such that $L(\T_{\{G\}}) \ge 0$ has an integral
representation. Let $S \colon (x_1,x_2,x_3,\ldots) \mapsto (x_2,x_3,x_4,\ldots)$ be the shift operator. Note that for every $n \in \NN$,
$S^n G (S^\ast)^n = A_n x^3-\Id x^2$ where
\beqn
A_n & = & \left(
\begin{array}{cccc}
\frac{1}{n+1}&0&0&\ldots\\
0&\frac{1}{n+2}&0&\ldots\\
0&0&\frac{1}{n+3}&\ldots\\
\vdots&\vdots&\vdots&\ddots\\
\end{array}
\right)
\preceq 
\frac{1}{n+1} \Id
\eeqn
Since $L(\T_{\{G\}}) \ge 0$, it follows that $L(A_n x^3)-L(\Id x^2)=L(S^n G (S^\ast)^n)\ge 0$ for every $n$.
By the Cauchy-Schwartz inequality, it follows that $0 \le L(\Id x^2) \le L(A_n x^3) \le L(A_n^2)^{1/2}L(\Id x^6)^{1/2}
\le \frac{1}{(n+1)} L(\Id)^{1/2}L(\Id x^6)^{1/2}$. In the limit, we get that $L(\Id x^2)=0$. Using Cauchy-Schwartz again,
we deduce that $L(x^{k} B_k)=0$ for every $k \in \NN$ and $B_k \in B(\ell^2)$. Therefore, for every $F=\sum_{k=0}^m x^k B_k$,
we have that $L(F)=L(B_0)=L\vert_{B(\ell^2)} (F(0))$. Therefore $L$ has a representing measure which assigns to the set $\{0\}$
the functional $L\vert_{B(\ell^2)}$.


\begin{thebibliography}{1}


\bibitem{Amb-Vas}
C.-G.~Ambrozie, F.-H.~Vasilescu, Operator-theoretic Positivstellensätze, Z.~Anal.~Anwend.~22 (2003), No. 2, 299--314.

\bibitem{Akh1}
N.I.~Akhiezer, The Classical Moment Problem, Oliver and Boyd, New York, 1965.

\bibitem{Akh}
N.I.~Akhiezer, I.~Glazman, Theory of Linear Operators in Hilbert Space, Dover Publications, New York, 1993.

\bibitem{Ath}
A.~Athavale, Holomorphic kernels and commuting operators, Trans.~Amer.~Math.~Soc.~304 (1987) 101--110.

\bibitem{Bar}
R.~G.~Bartle, N.~Dunford, J.~Schwartz. Weak compactness and vector measures, Canad.~J.~Math.~7 (1955) 289--305.

\bibitem{Ber}
S.~K.~Berberian, Notes on Spectral Theory, D. van Nostrand Company, Princeton, 1966.

\bibitem{Cim1}
J.~Cimprič, A representation theorem for archimedean quadratic modules on $^{\ast}$-rings, Can.~math.~bull.~52 (2009) 39--52.

\bibitem{Cim2}
J.~Cimprič, Strict Positivstellensätze for matrix polynomials with scalar constraints, Linear algebra appl., 434 (2011), iss. 8, 1879--1883.

\bibitem{Cim3}
J.~Cimprič, Real algebraic geometry for matrices over commutative rings, J.~Algebra 359 (2012), 89--103.

\bibitem{Cim}
J.~Cimprič, Archimedean operator-theoretic Positivstellensätze, J.~Funct.~Anal.~260 (2011), no.~10, 3132–-3145.

\bibitem{Con}
J.~Conway, A Course in Functional Analysis, Springer-Verlag, New York, 1990.

\bibitem{ds} 
H.~Dette, W.~J.~Studden, 
Matrix measures, moment spaces and Favard's theorem for the interval $[0,1]$ and $[0,\infty)$,
Linear Algebra Appl.~345 (2002), 169–-193. 

\bibitem{Dor}
I.~Dobrakov, On representation of linear operators on $C_0(T, X)$, Czechoslovak Math.~J.~21 (1971) 13--30.

\bibitem{Dor1}
I.~Dobrakov, On integration in Banach spaces I, Czechoslovak Math.~J.~20 (1970) 511--536.

\bibitem{dr}
M.~A.~Dritschel, J.~Rovnyak, The operator Fejér-Riesz theorem, 
Oper.~Theory Adv.~Appl.~207 (2010) 223--254.

\bibitem{ForStine}
W.~Forrest Stinespring, Positive functions on $C^{\ast}$-algebras, Proc.~Amer.~Math.~Soc.~6 (1955) 211--216.

\bibitem{HS} 
C.W.J.~Hol, C.W.~Scherer, Matrix sum-of-squares relaxations for robust semi-definite programs, Math.~Programming 107 (2006) 189--211.

\bibitem{jp}
T.~Jacobi, A.~Prestel, Distinguished representations of strictly positive polynomials,
J.~Reine Angew.~Math.~532 (2001) 223–-235.


\bibitem{John}
G.~W.~Johnson, The dual of C(S,F), Math.~Ann.~187 (1970) 1--8.

\bibitem{KS}
I.~Klep, M.~Schweighofer, Pure states, positive matrix polynomials and sums of hermitian squares, Indiana Univ.~Math.~J.~59 (2010),No. 3, 857--874.

\bibitem{kovalishina} 
I.~V.~Kovalishina,
Analytic theory of a class of interpolation problems,
Izv.~Akad.~Nauk SSSR Ser.~Mat.~47 (1983), no. 3, 455–-497. 

\bibitem{krein}
M.~Krein, Infinite J-matrices and a matrix-moment problem,
Doklady Akad.~Nauk SSSR (N.S.) 69 (1949) 125–-128.

\bibitem{Lax}
P.D.~Lax, Functional Analysis, John Wiley \& Sons, New York, 2002.

\bibitem{Li}
B.~Li, Real Operator Algebras, World Scientific Publishing, Singapore, 2003.

\bibitem{dl} A.J.~Durán, P.~López-Rodríguez, 
The matrix moment problem,
Margarita mathematica (2001) 333–-348.

\bibitem{Mar}
M.~Marshall, Positive Polynomials and Sums of Squares, American Mathematical Society, Providence, 2008.

\bibitem{Phelps}
R.~Phelps, Lectures on Choquet's Theorem, Springer-Verlag, Berlin, 2001.

\bibitem{Pow}
R.T.~Powers, Selfadjoint algebras of unbounded operators.~II, Trans.~Amer.~Math.~Soc.~187 (1974), 261--293.

\bibitem{put}
M.~Putinar, Positive polynomials on compact semi-algebraic sets, Indiana Univ.~Math.~J.~42 (1993), no. 3, 969–-984. 

\bibitem{ps1}
M.~Putinar, C.~Scheiderer,
Multivariate moment problems: Geometry and indeterminateness,
Ann.~Sc.~Norm.~Super.~Pisa Cl.~Sci.~ 5 (2006), no.~2, 137–-157. 

\bibitem{ps2}
M.~Putinar, K.~Schmüdgen, Multivariate determinateness, 
Indiana Univ.~Math.~J.~57 (2008), no.~6, 2931–-2968.

\bibitem{RieszNagy}
F.~Riesz, B.Sz.~Nagy, Functional Analysis, Blackie \& Son Limited, London, 1956.

\bibitem{rosenblatt}
M.~Rosenblatt, A multi-dimensional prediction problem,
Ark.~Mat.~3 (1958), 407–-424. 

\bibitem{rosenblum} 
M.~Rosenblum, 
Vectorial Toeplitz operators and the Fejér-Riesz theorem,
J.~Math.~Anal.~Appl.~23 (1968) 139-–147. 

\bibitem{Rud}
W.~Rudin, Functional Analysis, McGraw-Hill, New York, 1973.

\bibitem{Mit-You}
S.~K.~Mitter, S.~K.~Young, Integration with respect to operator-valued measures with applications to quantum estimation theory, 
\url{http://dspace.mit.edu/handle/1721.1/2845}.

\bibitem{Sch}
K.~Schm\" udgen, On a generalization of the classical moment problem, J.~Math.~Anal.~Appl.~125 (1987) 461--470.

\bibitem{sch-psatz}
K.~Schm\" udgen, The K-moment problem for compact semi-algebraic sets. Math.~Ann.~289 (1991), no. 2, 203–-206.

\bibitem{schbook} K. Schm\" udgen. \textit{Unbounded operator algebras and representation theory.}
Operator Theory: Advances and Applications, 37. Basel etc.: Birkh\" auser Verlag. 1989.


\bibitem{sav-sch} 
Y.~Savchuk, K.~Schm\"udgen, Positivstellens\"atze for Algebras of Matrices,
arXiv:1004.1529 (Preprint, April 2010).

\bibitem{Vas1}
F.-H.~Vasilescu, Subnormality and moment problems, Extracta Math. 24 (2009) 167--186.

\bibitem{Vas2}
F.-H.~Vasilescu, Spectral measures and moment problems, in the volume \textit{Spectral theory and its applications
} (2003) 173--215.

\bibitem{Zag}
S.M.~Zagorodnyuk, The matrix Stieltjes moment problem: a
description of all solutions, \url{http://arxiv.org/pdf/1002.4511.pdf}, 24.8.2012.


\bibitem{Zag1}
S.M.~Zagorodnyuk, A description of all solutions of the matrix
Hamburger moment problem in a general case, \url{http://arxiv.org/pdf/1002.4511.pdf}, 24.8.2012.


\end{thebibliography}
\end{document}